\documentclass[12pt, a4paper, reqno]{amsart}
\usepackage[T1]{fontenc}
\usepackage[english]{babel}
\usepackage{amsfonts}
\usepackage{amsmath}
\usepackage{amssymb}
\usepackage{amsthm}
\usepackage{thmtools}
\usepackage[mathscr]{euscript}
\usepackage{booktabs}
\usepackage{enumitem}
\usepackage{mathtools}
\usepackage{graphicx}
\usepackage{tikz,pgf}
\usepackage{hyperref}
\hypersetup{
    colorlinks,
    linkcolor={black},
    citecolor={black},
    urlcolor={black}
}
\usepackage{varioref}
\usepackage[nameinlink]{cleveref}
\usepackage{orcidlink}
\theoremstyle{plain}
\newtheorem{lemma}{Lemma}[section]
\newtheorem{proposition}[lemma]{Proposition}
\newtheorem{theorem}[lemma]{Theorem}
\newtheorem{corollary}[lemma]{Corollary}

\newtheorem*{conjecturenonum}{\textbf{Conjecture}}
\theoremstyle{definition}
\newtheorem{definition}[lemma]{Definition}

\newtheorem{remark}[lemma]{Remark}

\newcommand{\defeq}{\coloneq}
\newcommand{\st}{\,\,s.t.\,\,}

\newcommand{\R}{\mathbb{R}}

\date{}

\begin{document}

\title{Highly regular vertex-transitive graphs are globally rigid}
\author{Angelo El Saliby}
\address[Angelo El Saliby \orcidlinkc{0009-0007-3612-5648}]{Max Planck
Institute for Mathematics in the Sciences Leipzig, Germany }
\email{angelo.el.saliby@mis.mpg.de}


\begin{abstract}
A graph is said to be globally rigid in $d$-dimensional space
if almost all of its embeddings are unique up to isometries.
If a graph has enough automorphisms to send any of its vertices into any other, then it is called vertex-transitive.
We show that, in any dimension, highly regular vertex-transitive graphs are globally rigid,
positively answering a conjecture of Sean Dewar.
Furthermore, we construct examples that show that our constant for regularity is best possible.
\end{abstract}

\maketitle

\section{Introduction}

        Given a graph and an assignment of a point in $d$-dimensional space for every one of its vertices,
        graph rigidity asks whether there are finitely or infinitely many other equivalent realizations up to isometry,
        i.e. other assignments of points in $d$-dimensional space that preserve the edge lengths of edges in the graph,
        but that cannot be mapped one into another through an isometry of the ambient space. 
        When the starting realization is \emph{generic} this is a property of the graph itself,
        and we say that a graph is rigid if there are finitely many equivalent realizations to any given generic one,
        again up to isometry.
        If all generic realizations are unique up to isometry, we say that the graph is globally rigid.
        Efficient combinatorial characterizations of rigidity and global rigidity exist in dimension $1$ and $2$, 
        while finding combinatorial characterizations in dimension $d\geq 3$ is one of the main open problems in rigidity theory.
        As such, many results and classifications are known in the plane but not in higher dimensions.
        One such example arises in the study of vertex-transitive graphs,
        which are graphs whose local structure at each vertex is the same.
        Vertex-transitive graphs are regular,
        meaning that all the vertices have the same degree, 
        and it was proved in \cite{jackson2007} that a vertex-transitive graph of degree at least $6$ is globally rigid in the plane. 
        This leads to the following conjecture.
        \begin{conjecturenonum}[{\cite[Conjecture~20]{dewar2023}}]
	    There exists a function $f \colon \mathbb{N}\rightarrow \mathbb{N}$ with the following property.
        For $d\in \mathbb{N}$, if $G$ is a vertex-transitive graph of degree at least $f(d)$, then $G$ is globally rigid in $\mathbb{R}^d$.
        \end{conjecturenonum}
        Proving that a highly vertex-connected graph is globally rigid 
        was a decades long open conjecture in graph rigidity theory
        until \cite{villanyi2023}.
        There, the proof builds on recent developments on \emph{weakly globally linked pairs} from \cite{jordan2024}.
         Although a connected vertex-transitive graph might have a higher degree than its vertex-connectivity,
        it's vertex-connectivity is at least two thirds of its degree. 
        As an immediate consequence one can verify Dewar's conjecture.
        In this paper, we improve on this result by providing a tight constant:
        \begin{theorem}\label{theorem:main-theorem-in-introduction}
            If $G$ is a connected vertex-transitive graph of degree at least $d(d+1)$, 
            then $G$ is globally rigid in $\mathbb{R}^d$. 
            Moreover, there exist connected vertex-transitive graphs of degree $d(d+1)-1$ that are not globally rigid in $\mathbb{R}^d$.
        \end{theorem}
        To prove the result for the improved constant, 
        we employ the same argument as in \cite{villanyi2023} with minor adaptations.
        To show that the constant we provide is best possible we generalize a result of \cite{jackson2007}.

        The paper is organized as follows. 
        In \Cref{section:weakly-globally-linked-pairs} we briefly introduce the rigidity matroid 
        and we recall some important definitions and results on weakly globally linked pairs from \cite{jordan2024}. 
        Furthermore, we discuss ordering-induced subgraphs and some criteria for their independence in the rigidity matroid,
         summarizing results from \cite{villanyi2023}. 
         In \Cref{section:vertex-transitive-graphs} we prove the main result of this paper.

\section{Weakly globally linked pairs}\label{section:weakly-globally-linked-pairs}
We briefly summarize the notions and results from \cite{villanyi2023} and \cite{jordan2024} that we need. 
We start by giving the usual definitions in graph rigidity. More details can be found in \cite{graver1993, jordan2016}.

Let $G=(V, E)$ be a simple, undirected graph on vertex set $V$ and edge set~$E$. 
If $v,w\in V$ are the endpoints of an edge in $E$, we write $vw\in E$. 
Fix $d\in\mathbb{N}$ such that $d\geq 1$. A function \begin{equation*}
p \colon V\rightarrow \mathbb{R}^d
\end{equation*}
    is called \emph{realization} of $G$ in $\mathbb{R}^d$.
    Such a realization can also be viewed as a point in $\bigl(\mathbb{R}^d\bigr)^V$.
    Define the \emph{rigidity map} $f_G \colon \bigl(\mathbb{R}^d\bigr)^V\rightarrow \mathbb{R}^E$ by \begin{equation*}
    \bigl(p_v\bigr)_{v\in V}\mapsto \bigl(\|p_v-p_w\|^2\bigr)_{vw\in E}.
    \end{equation*}
    Let $K_V$ denote the complete graph on $V$. We say that two realizations~$p$ and~$q$ of $G$ are \emph{equivalent} if $f_G(p)= f_G(q)$,
     while we say they are \emph{congruent} if $f_{K_V}(p) = f_{K_V}(q)$. 
     We say that the graph $G$ is \emph{globally rigid} if there exists an open dense subset $U$ of $\bigl(\mathbb{R}^d\bigr)^V$ such that for all $p\in U$ we have
     \begin{equation}
     f_G^{-1}\bigl(f_G(p)\bigr) = f_{K_V}^{-1}\bigl(f_{K_V}(p)\bigr).
    \end{equation}
The rigidity matroid on a realization $p$ of $K_V$ is the row matroid 
associated to the Jacobian of $f_{K_V}$, evaluated at $p$.
It can be shown that there exists an open dense subset of $\bigl(\mathbb{R}^d\bigr)^V$ such that the corresponding matroids are all isomorphic.
We call the arising matroid the \emph{rigidity matroid} of $V$ in $\mathbb{R}^d$, 
and we call any realization in this open dense subset \emph{generic}. 
We say that a graph $G=(V, E)$ is \emph{rigid} in $\mathbb{R}^d$ if $E$ has maximal rank in the rigidity matroid of $V$ in $\mathbb{R}^d$.

\subsection{Weakly globally linked pairs}
Most of the results and definitions in this section are summarized from \cite{jordan2024},
but the exposition and the selection of them is strongly reliant on \cite{villanyi2023}.
We refer to these papers for further details.
Weakly globally linked pairs adapt the combinatorial aspects of linked pairs to the global rigidity setting.
They were introduced in \cite{jordan2024} and they are a crucial ingredient in \cite{villanyi2023}.
\begin{definition}\label{definition:globally-linked-pairs-in-frameworks}
    Fix $d\in\mathbb{N}_{>0}$. Let $G =(V, E)$ be a graph and consider two distinct vertices $u, v\in V$.
    Let $p$ be a realization of $G$ in $\mathbb{R}^d$. We call the tuple $(G, p)$ a \emph{framework} (in $\mathbb{R}^d$).
    We say that the pair $u, v$ is \emph{globally linked} in $(G, p)$ in $\mathbb{R}^d$ if the set \begin{equation*}
        \Bigl\{\bigl\|q(v)-q(u)\bigr\|\text{ s.t.\ the framework } (G, q) \text{ is equivalent to } (G, p)\Bigr\}
    \end{equation*}
    contains only one element.
\end{definition}

Hence, a framework is globally rigid if and only if every pair of its vertices is globally linked. 
Global linkedness of pairs is not a generic property, 
in the sense that a pair can be globally rigid in a generic realization without being so in \emph{all} generic realizations.
This warrants the following definition.

\begin{definition}\label{definition:weakly-globally-linked}
    Fix $d\in\mathbb{N}_{>0}$. 
    A pair of vertices $u,v\in G$ is \emph{weakly globally linked} in $\R^d$ 
    if there exists a generic $d$-realization $p$ of $G$ such that $u$ and $v$ are globally linked in $(G,p)$.
    If $u$ and $v$ are not weakly globally linked, they are called \emph{globally loose}.
\end{definition}

Despite the previous remark, 
the following theorem shows that global rigidity is equivalent to every pair of vertices being weakly globally linked.
\begin{theorem}[{\cite[Lemma~3.2]{jordan2024}}] \label{theorem:global-rigidity-through-weak-global-linkedness}
Fix $d\in\mathbb{N}_{>0}$ and let $G=(V,E)$ be a graph. 
Then $G$ is globally rigid in $\R^d$ if and only if 
every pair of vertices of~$G$ is weakly globally linked in $\mathbb{R}^d$.
\end{theorem}

Weak global linkedness of pairs is a combinatorial property 
and is hence invariant under automorphisms of the graph,
as we make precise in the following lemma. 
The proof is immediate and is hence omitted.
\begin{lemma}\label{lemma:weak-global-linkedness-is-orbit-closed}
    Fix $d\in\mathbb{N}_{>0}$ and let $G=(V, E)$ be a graph.
     Let 
     \begin{equation*}
        W_d(G)\defeq \left\{\begin{array}{ll}
                             \{u,v\}\st
        u, v\in V\text{ and }
        u, v \text{ are } \\\text{weakly globally linked in $G$ in $\mathbb{R}^d$}
                            \end{array}
\right\}.
    \end{equation*}
    For any $e\in E(K_V)$ denote by~$O_G(e)$ the orbit of $e$ under the automorphisms group of $G$.
    Then for any $e\in E(K_V)$ we have that \begin{equation*}
        e\in W_d(G)\iff O_G(e)\subseteq W_d(G).
    \end{equation*}
\end{lemma}

The following result is contained in the proof of \cite[Theorem~1.2]{villanyi2023}.
\begin{lemma}[\cite{villanyi2023}]\label{lemma:maximal-neighbour-cliques-when-all-weakly-globally-linked-pairs-are-adjacent}
    Fix $d\in\mathbb{N}$ with $d\geq 2$ and let $G=(V,E)$ be a graph.
    Assume that $G$ is vertex-$(d+1)$-connected and that all weakly globally linked pairs in $\mathbb{R}^d$ are adjacent in $G$.
    Then for every vertex $v\in V$,
    if $V_1, V_2\subseteq N_G(v)$ are distinct maximal cliques in $G\bigl[N_G(v)\bigr]$,
    we have that $|V_1\cap V_2| < d-1$.
\end{lemma}

\subsection{Ordering-induced subgraphs}

Ordering induced subgraphs were introduced in \cite{villanyi2023} and 
they play a central role in the proof of the main theorems of that paper. 
Here, we briefly recall the construction and state the main results from \cite{villanyi2023} that we need. 
The main result asserts that if in graph all weakly globally linked pairs are adjacent,
then all ordering-induced subgraphs are independent in the rigidity matroid (\Cref{lemma:indepndence-of-pi-construction-globally-rigid-case}).

\begin{definition}\label{definition:ordering-of-graph}
    Let $V$ be a finite set of cardinality $n$.
    An \emph{ordering} of~$V$ is an ordered tuple $\pi = (v_1, \dots, v_n)$ such that $V= \{v_1, \dots, v_n \}$.
    To every ordering $\pi$, we associate an order relation on $V$, denoted $<_{\pi}$.
    For every $u, w\in V$, the order relation $<_{\pi}$ is defined by \begin{equation*}
        v<_{\pi} w \iff \exists i,j\in \{1,\dots, n\} \st v= v_i, w= v_j \text{ and } i<j.
    \end{equation*}
    Let $G=(V, E)$ be a graph and let $\pi = (v_1, \dots, v_n)$ be an ordering of~$V$.
    For any $v\in V$, we define the \emph{ordered parents} of $v$ in $G$ as  \begin{equation*}
        N_{G, \pi}(v)\defeq N_G(v)\cap \{w\in V\st w<_{\pi} v\}.
    \end{equation*}
    Furthermore, we define the \emph{ordered degree} of $v$ as
    \begin{equation*}
        \deg_{G, \pi}(v)\defeq \bigl|N_{G, \pi}(v)\bigr|.
    \end{equation*}
\end{definition}

\begin{definition}\label{definition:pi-subgraph-in-Rd}
Fix $d\in\mathbb{N}$ with $d\geq 2$.
Let $G=(V, E)$ be a graph and let $\pi$ be an ordering of~$V$.
We define a family of subgraphs of~$G$ induced by $\pi$.
We call any of these subgraphs a \emph{$\pi$-subgraph} of $G$ in~$\mathbb{R}^d$.
Let $n\defeq |V|$. A $\pi$-subgraph of $G$ is a graph of the form $G_{\pi}^n$,
where~$G_{\pi}^n$ is iteratively constructed as follows.
Set $G_\pi^0 \defeq (\emptyset,  \emptyset)$.
For any $i\in\{0,\dots, n-1\}$,
to obtain $G_{\pi}^{i+1}$ from~$G_{\pi}^i$, let $H^i\defeq G_\pi^{i}+v_i$ and consider $N_{G, \pi}(v_i)$ and $\deg_{G, \pi}(v_i)$.
Then:
\begin{enumerate}
            \item if $\deg_{G, \pi}(v_i)\leq d$, the graph $G_{\pi}^{i+1}$ is obtained from $H^i$ by connecting $v_i$ to all vertices in $N_{G, \pi}(v_i)$;
            \item if $\deg_{G, \pi}(v_i)>d$ and $N_{G,\pi}(v_i)$ induces a clique in $G$,
            the graph $G_{\pi}^{i+1}$ is obtained from $H^i$ by connecting $v_i$ to $d$ vertices chosen from $N_{G, \pi}(v_i)$;
            \item if $\deg_{G, \pi}(v_i)>d$ and $N_{G,\pi}(v_i)$ does not induces a clique in $G$,
            the graph $G_{\pi}^{i+1}$ is obtained from $H^i$ by connecting $v_i$ to $d+1$ distinct vertices chosen in $N_{G, \pi}$,
            subject to the constraint that at least two of those $d+1$ vertices are not adjacent in $G$.
        \end{enumerate}
A graph that is isomorphic to a $\pi$-subgraph of $G$, for some ordering $\pi$ of $V$, is called \emph{ordering-induced} subgraph of $G$.
\end{definition}

The following two results are at the heart of the proofs of the main theorems of \cite{villanyi2023}. In the same way, they will ultimately allow us to reach a contradiction in the proof of \Cref{theorem:highly-regular-vertex-transitive-graphs-are-globally-rigid}.
The statement of \Cref{theorem:expected-edges-of-pi-construction-for-low-clique-graphs} is slightly different and it is obtained by applying Maxwell's necessary condition to \cite[Lemma~3.2]{villanyi2023}.

\begin{lemma}[{\cite[Lemma~4.1]{villanyi2023}}]\label{lemma:indepndence-of-pi-construction-globally-rigid-case}
    Fix $d\in\mathbb{N}$ with $d\geq 2$, let $G=(V, E)$ be a graph and let $\mathcal{R}_d$ be the rigidity matroid on $V$ in $\mathbb{R}^d$.
    Assume that all weakly globally linked pairs in $G$ are also adjacent in $G$,
    then any ordering-induced subgraph of $G$ is $\mathcal{R}_d$-independent.

\end{lemma}

\begin{lemma}[{\cite[Lemma~3.2]{villanyi2023}}]\label{theorem:expected-edges-of-pi-construction-for-low-clique-graphs}
Fix $d\in\mathbb{N}$ with $d\geq 2$. Let $G=(V, E)$ be a graph such that, for every $v\in V$, we have:
\begin{enumerate}
    \item\label{item:minimal-degree-probabilistic-lemma} $\deg_G(v)\geq d(d+1)$;
    \item\label{item:no-neighborhood-cliques-probabilistic-lemma} $N_G(v)$ is not a clique;
    \item\label{item:maximal-cliques-in-neighborhood-probabilistic-lemma} if $H_1, H_2\subseteq N_G(v)$ are vertex sets of two  maximal cliques in $G\bigl[N_G(v)\bigr]$
    and $H_1\neq H_2$ then $|H_1\cap H_2|\leq d-2$.
\end{enumerate}
Then at least one ordering-induced subgraph of $G$ is not $\mathcal{R}_d$-independent.
\end{lemma}

\section{Vertex-transitive graphs}\label{section:vertex-transitive-graphs}

The paper \cite{jackson2007} provides a characterization of vertex-transitive globally rigid graphs in $\mathbb{R}^2$,
and in particular it shows that every vertex-transitive graph of degree at least~$6$ is globally rigid in the plane.
In this section, we generalize this result and show that a vertex-transitive graph of regularity at least $d(d+1)$ is globally rigid in $\mathbb{R}^d$ (\Cref{theorem:highly-regular-vertex-transitive-graphs-are-globally-rigid}).
Moreover, we construct a connected vertex-transitive graph of degree $d(d+1)-1$ that is not globally rigid in $\mathbb{R}^d$ (\Cref{corollary:example-tightness-contstant}),
showing that $d(d+1)$ is the best constant. Together, these facts imply \Cref{theorem:main-theorem-in-introduction}.
The papers \cite{jackson2007} and \cite{dewar2023} go further in the analysis in the plane
by characterizing when a vertex-transitive graph of degree less than $6$ is globally rigid.
We leave this question to future work.

We start by recalling the definition of this family of graphs.

\begin{definition}
 A graph $G=(V,E)$ is \emph{vertex-transitive} if for every pair of distinct vertices $v,w\in V$ there exists an automorphism $\phi$ of  $G$ such that $\phi(v)=w$.
\end{definition}

One key result about vertex-transitive graphs is that, if they are connected, then they are highly vertex-connected.
\begin{theorem}[{\cite[Theorem~3]{watkins1970}}]\label{lemma:vertex-transitive-implies-regular-and-highly-connected}
    Let $G=(V, E)$ be a connected vertex-transitive graph of degree $\omega$.
    Let $k\in\mathbb{N}_{>0}$ be the maximal number such that $G$ is vertex-$k$-connected, then
    \begin{equation*}
        k \geq \frac{2}{3} \omega.
    \end{equation*}
\end{theorem}
In \cite{villanyi2023}, Vill\'anyi proved that a vertex-$d(d+1)$-connected graph is globally rigid in $\mathbb{R}^d$, thus verifying Dewar's conjecture.
\begin{corollary}\label{corollary:dewar-conjecture-through-villanyi-result}
    Fix $d\in\mathbb{N}$ with $d\geq 1$.
    If $G$ is a connected vertex-transitive graph of regularity at least $\frac{3}{2}d(d+1)$ then $G$ is globally rigid in $\mathbb{R}^d$.
\end{corollary}

The main result of this paper improves on this constant.

\begin{theorem}\label{theorem:highly-regular-vertex-transitive-graphs-are-globally-rigid}
	Fix $d\in\mathbb{N}_{>0}$ and let $G$ be a connected vertex-transitive graph of degree at least $d(d+1)$. Then $G$ is globally rigid in $\mathbb{R}^d$.
\end{theorem}

The following is an easy observation that is nevertheless essential in the proof of the theorem.
\begin{lemma}\label{lemma:all-neighbors-are-cliques-implies-complete-connected-components}
    Let $G=(V, E)$ be a graph and assume that for every $v\in V$ the set $N_G(v)$ induces a clique in $G$.
    Then every connected component of $G$ is a complete graph.
\end{lemma}

\begin{proof}[Proof of \Cref{theorem:highly-regular-vertex-transitive-graphs-are-globally-rigid}]
    Let $V$ be a finite set.
    Suppose for a contradiction that there exists a vertex-transitive graphs on $V$ of degree at least $d(d+1)$ that is not globally rigid in~$\mathbb{R}^d$.
    Since there are finitely many graphs on $V$, we can take one such graph that has the maximal number of edges. We denote it $G=(V, E)$.
    Notice that our assumption on the degree of $G$ implies that \Cref{item:minimal-degree-probabilistic-lemma} of \Cref{theorem:expected-edges-of-pi-construction-for-low-clique-graphs} holds for $G$.

    First, we show that all weakly globally linked pairs of vertices in $G$ are adjacent.
    Indeed, if for a contradiction $u,v\in V$ are weakly globally linked but not adjacent,
    then every edge in the orbit of $\{u, v\}$ under the automorphisms group of $G$, that we denote by $O_G(e)$, is weakly globally linked (\Cref{lemma:weak-global-linkedness-is-orbit-closed}).

    Since $G$ is not globally rigid,
    by \Cref{theorem:global-rigidity-through-weak-global-linkedness} neither $G+O_{G}(e)$ is globally rigid.
    However, the graph $G+O_{G}(e)$ is again a vertex-transitive graph on $V$ of degree at least $d(d+1)$,
    but with more edges than~$G$.
    This contradicts the maximality assumption on the cardinality of~$E$, and proves that all weakly globally linked pairs of vertices in $G$ are adjacent. By \Cref{lemma:indepndence-of-pi-construction-globally-rigid-case}, we conclude that all ordering induced subgraphs of $G$ are independent in the rigidity matroid.

    Secondly, for every $v\in V$, the set $N_G(v)$ does not induce a clique.
    Indeed, since $G$ is vertex-transitive, if one neighborhood induces a clique, all neighborhoods induce a clique and $G$ is a complete graph (\Cref{lemma:all-neighbors-are-cliques-implies-complete-connected-components}).
    Since complete graphs are globally rigid, this would yield a contradiction.
    We have proved that \Cref{item:no-neighborhood-cliques-probabilistic-lemma} of \Cref{theorem:expected-edges-of-pi-construction-for-low-clique-graphs} holds for~$G$.
    Finally, since all weakly globally linked pairs are adjacent in $G$ in~$\mathbb{R}^d$,
    we have that all ordering-induced subgraphs of~$G$ are independent in $\mathbb{R}^d$ (\Cref{lemma:indepndence-of-pi-construction-globally-rigid-case}).
    The fact that all weakly globally linked pairs are adjacent also implies, by \Cref{lemma:maximal-neighbour-cliques-when-all-weakly-globally-linked-pairs-are-adjacent},
    that \Cref{item:maximal-cliques-in-neighborhood-probabilistic-lemma} of \Cref{theorem:expected-edges-of-pi-construction-for-low-clique-graphs} holds for $G$.
    Notice that we can apply \Cref{lemma:maximal-neighbour-cliques-when-all-weakly-globally-linked-pairs-are-adjacent} because the vertex connectivity of $G$ is greater than \begin{equation*}
        \frac{2}{3} d(d+1) \geq d+1,
    \end{equation*}
    assuming that $d>1$ (\Cref{lemma:vertex-transitive-implies-regular-and-highly-connected}).

    On the one hand, we have proved that all ordering-induced subgraphs of~$G$ are independent in $\mathbb{R}^d$. On the other hand, we have verified that the conditions of \Cref{theorem:expected-edges-of-pi-construction-for-low-clique-graphs} are satisfied by $G$, which implies in particular that at least one ordering-induced subgraph of $G$ is dependent in the rigidity matroid, a contradiction.
\end{proof}

\begin{remark}
	The analogous result of \Cref{theorem:highly-regular-vertex-transitive-graphs-are-globally-rigid} for Cayley graphs can also be proved independently, by considering the orbit of an edge under the action of the corresponding group and proving that adding such an orbit yields yet another Cayley graph.
\end{remark}

\begin{remark}
 The outline contained in the proof of \Cref{theorem:highly-regular-vertex-transitive-graphs-are-globally-rigid} is essentially identical to the outline of the proof of the main theorem of \cite{villanyi2023}. What we have done is simply observe that the main steps also work when dealing with vertex-transitive graphs.
\end{remark}

\begin{remark}
	The lexicographic product of two graphs is a new graph whose vertex set is the Cartesian product of the vertex sets of the starting graphs. If $G$ and $H$ are graphs, a pair of vertices $(g, h), (g', h')\in V(G)\times V(H)$ is adjacent in the lexicographic product of $G$ and $H$ if either $\{g, g'\} \in E(G)$ or $g=g'$ and  $\{h, h'\} \in E(H)$.
By considering the lexicographic product between circuits and complete graphs, it is possible to find a family of vertex-transitive graphs whose global rigidity is implied by \Cref{theorem:highly-regular-vertex-transitive-graphs-are-globally-rigid} but not by the results in \cite{villanyi2023}.
\end{remark}

To prove the tightness of the constant for vertex-transitive graphs, we can generalize \cite[Theorem~2.2.d]{jackson2007}.
\begin{proposition}\label{proposition:tightness-vertex-transitive}
Fix $d\in\mathbb{N}$ with $d\geq 2$ and let $k\defeq d(d+1)-1$. Let $G$ be a $k$-regular vertex-transitive graph and assume it contains a subgraph $H$ such that:
\begin{enumerate}
    \item $H$ is spanning, namely $V(H) = V(G)$;
    \item there exists $s\in\mathbb{N}$ with $s\geq d(d+1)$, and such that $H$ consists of $s$ disjoint copies of a complete graph on $k$ vertices.
\end{enumerate}
Then $G$ is not globally rigid in $\mathbb{R}^d$.
\end{proposition}
In the proof we will need the following lemma.
\begin{lemma}[{\cite[Lemma~2.2]{villanyi2023}}]\label{lemma:cardinality-upper-bound-thorugh-partition-submodular-function}
	Fix $d\in\mathbb{N}$ with $d\geq 2$ and let $G=(V, E)$ be a graph. Let $r$ be the rank function of the rigidity matroid on $V$ in $\mathbb{R}^d$. If $E_0, \dots, E_k \subseteq E$ are such that $E=\cup_{i=0}^{k} E_i$ and $\bigl|V(E_i)\bigr|\geq d+1, \, \forall i = 0,\dots, k$ then \begin{equation*}
		r(E) \leq |E_0| + \sum_{i=1}^{k} \left(d |V(E_i)| - \binom{d+1}{2}\right)
	\end{equation*}
\end{lemma}

\begin{proof}[Proof of \Cref{proposition:tightness-vertex-transitive}]
    Let $K^1, \dots, K^s$ be complete graphs on $k$ vertices such that \begin{equation*}
        H =\bigcup_{i=1}^s K^i
     \end{equation*}
    Notice that $H$ has $ks$ vertices and it is regular of degree $k-1$.
    Moreover, since $H$ is spanning in $G$, also $G$ has $ks$ vertices.
    Let $E_0$ be the set of edges of $G$ that are not contained in $H$.
    Since each vertex is in some copy $K^i$ and $G$ has degree one more than the degree of $H$,
    each vertex of $G$ is adjacent to exactly one edge in $E_0$.
    This means that
    \begin{equation}\label{equation:tightness-vertex-transitive-eq1}
        |E_0|= \frac{|V(G)|}{2} = \frac{ks}{2}.
    \end{equation}
    Moreover, the edge sets
    \begin{equation*}
        E^i \defeq E(K^i), \forall i\in\{1,\dots, s\},
    \end{equation*}
    together with $E_0$, partition the edge set of $G$. Let $e\in E_0$ and $E_0'\defeq E_0\setminus\{e\}$. We now prove that $G-e$ is not rigid.
    Let $r$ be the rank function of the rigidity matroid on $V$ in~$\mathbb{R}^d$.
    By \Cref{lemma:cardinality-upper-bound-thorugh-partition-submodular-function} we obtain \begin{equation}\label{equation:tightness-vertex-transitive-eq2}
        r\bigl(E\setminus\{e\}\bigr) \leq |E_0'| + \sum_{i=1}^{s} \left(d \cdot \bigl|V(E_i)\bigr| - \binom{d+1}{2}\right).
    \end{equation}
    Notice that to apply the lemma we used the fact that $\bigl|V(E_i)\bigr| \geq k\geq d+1$. Combining \Cref{equation:tightness-vertex-transitive-eq1,equation:tightness-vertex-transitive-eq2}, we obtain
    \begin{align}\label{equation:tightness-vertex-transitive-eq4}
        r\bigl(E\setminus\{e\}\bigr)    &\leq \frac{ks}{2} - 1 + s\left( d k - \binom{d+1}{2}\right)\nonumber\\
                & \leq ks\left(d+\frac{1}{2} - \frac{d(d+1)}{2k} -\frac{1}{ks}\right)
    \end{align}
    On the other hand, we know that $G-e$ is rigid if and only if \begin{equation*}
        r\bigl(E\setminus\{e\}\bigr) = d \bigl|V(G)\bigr| -\binom{d+1}{2} = dks -\binom{d+1}{2}.
    \end{equation*}
    However, since  \begin{align*}
        & dks -\binom{d+1}{2} -ks\left(d+\frac{1}{2} -\frac{d(d+1)}{2k} -\frac{1}{ks}\right) \geq \\
        &\geq  ks\left( d-d-\frac{1}{2}+\frac{d(d+1)}{2k}+\frac{1}{ks}\right) -\frac{d+1}{2}\\
        &\geq  ks \left(\frac{d(d+1)-k}{2k} +\frac{1}{ks}\right) -\binom{d+1}{2} \\
        & \geq \frac{1}{2}s-\binom{d+1}{2}+\frac{1}{2}>0,
    \end{align*}
    as soon as $s\geq d(d+1)$, we conclude that the upper bound of \Cref{equation:tightness-vertex-transitive-eq4} is strictly smaller than the required rank for rigidity.

    We have shown that $G-e$ is not rigid, which implies that $G$ is not redundantly rigid. By Hendrickson's necessary conditions for global rigidity,  we conclude that $G$ is not globally rigid in $\mathbb{R}^d$.
\end{proof}

\begin{corollary}\label{corollary:example-tightness-contstant}
    For all $d\in\mathbb{N}$ with $d\geq 2$, there exists a connected
vertex-transitive graph of degree $d(d+1)-1$ that is not globally rigid
in $\mathbb{R}^d$.
\end{corollary}
\begin{proof}
 Fix $d\in\mathbb{Z}$ with $d\geq 2$. We illustrate an example in $\mathbb{R}^d$ arising from \Cref{proposition:tightness-vertex-transitive}. Let $s\defeq d(d+1)$ and $k\defeq s-1$. Define \begin{equation*}
        V\defeq \mathbb{Z}_{s}\times \mathbb{Z}_{k},
    \end{equation*}
    and consider the graph $G'$ on vertex set $V$, whose edges are \begin{equation*}
        \bigl\{(i, j),\,\,(j+1, i)\bigr\}, \forall i\in\{1,\dots ,k\} \text{ and } \forall j\in \{i,\dots,  k\}.
    \end{equation*}
    For every $i\in\mathbb{Z}_{s}$, let $K^i$ be the complete graph on $\{i\}\times \mathbb{Z}_{k}$. Finally, define \begin{equation*}
        G\defeq G'+\bigcup_{i\in\mathbb{Z}_{s}} E(K^i).
    \end{equation*}
  The subgraph \begin{equation*}
        \bigcup_{i\in\mathbb{Z}_{s}} K^i
    \end{equation*}
    is the disjoint union of $s$ complete graphs on $k$ vertices, and it is spanning. Since $G$ is also vertex-transitive, by \Cref{proposition:tightness-vertex-transitive} we conclude that $G$ is not globally rigid in $\mathbb{R}^d$.
\end{proof}
\section*{Acknowledgements}
I proved these results while working on my Master's thesis, under the supervision of Matteo Gallet. I am grateful to Matteo for introducing me to the subject, for guiding me through my first research journey and for all the meaningful discussions around this manuscript in particular. I am also very grateful to Sean Dewar for the invaluable feedback they gave me both during the thesis work and on this particular manuscript.

\bibliographystyle{alphaurl}

\bibliography{biblio}

@Article{jordan2024,
 Author = {Jord{\'a}n, Tibor and Vill{\'a}nyi, Soma},
 Title = {Globally linked pairs of vertices in generic frameworks},
 Journal = {Combinatorica},
 ISSN = {0209-9683},
 Volume = {44},
 Number = {4},
 Pages = {817--838},
 Year = {2024},
 DOI = {10.1007/s00493-024-00094-3}
}

@article{villanyi2023,
  author = {Vill{\'a}nyi, Soma},
  title = {Every {{\(d(d + 1)\)}}-connected graph is globally rigid in
{{\(\mathbb{R}^d\)}}},
  journal = {Journal of Combinatorial Theory. Series B},
  issn = {0095-8956},
  volume = {173},
  pages = {1--13},
  year = {2025},
  language = {English},
  doi = {10.1016/j.jctb.2025.01.005}
}

@InCollection{jordan2016,
 Author = {Jord{\'a}n, Tibor},
 Title = {Combinatorial rigidity. {Graphs} and matroids in the theory of rigid frameworks},
 BookTitle = {Discrete geometric analysis},
 ISBN = {978-4-86497-035-8},
 Pages = {33--112},
 Year = {2016},
 Publisher = {Tokyo: Mathematical Society of Japan}
}

@book{graver1993,
 author = {Graver, Jack and Servatius, Brigitte and Servatius, Herman},
 title = {Combinatorial rigidity},
 series = {Graduate Studies in Mathematics},
 issn = {1065-7338},
 volume = {2},
 isbn = {0-8218-3801-6},
 year = {1993},
 publisher = {Providence, RI: American Mathematical Society},
 language = {English},
 
}

@article{jackson2007,
 author = {Jackson, Bill and Servatius, Brigitte and Servatius, Herman},
 title = {The 2-dimensional rigidity of certain families of graphs},
 journal = {Journal of Graph Theory},
 issn = {0364-9024},
 volume = {54},
 number = {2},
 pages = {154--166},
 year = {2007},
 language = {English},
 doi = {10.1002/jgt.20196}
}

@article{dewar2023,
 author = {Dewar, Sean},
 title = {Classifying the globally rigid edge-transitive graphs and distance-regular graphs in the plane},
 journal = {Journal of Graph Theory},
 issn = {0364-9024},
 volume = {103},
 number = {2},
 pages = {175--185},
 year = {2023},
 language = {English},
 doi = {10.1002/jgt.22913}
}

@article{watkins1970,
title = {Connectivity of transitive graphs},
journal = {Journal of Combinatorial Theory},
volume = {8},
number = {1},
pages = {23-29},
year = {1970},
issn = {0021-9800},
doi = {https://doi.org/10.1016/S0021-9800(70)80005-9},
author = {Mark E. Watkins}
}

\end{document}